\documentclass[10pt]{amsart}
\usepackage{max-paper}
\usetikzlibrary{backgrounds}

\newcommand{\stack}[5]{
    \begin{tikzpicture}[scale=0.6]
        \draw[thick] (-0.5,0) -- (1,0) -- (1,-2) -- (2,-2) -- (2,0) -- (3.5,0);
        \node[fill = white, draw = white] at (2.8,.5) {$#5$};
        \node[fill = white, draw = white] at (1.5,-1.5) {$#4$};
        \node[fill = white, draw = white] at (1.5,-.8) {$#3$};
        \node[fill = white, draw = white] at (1.5,-.1) {$#2$};
        \node[fill = white, draw = white] at (.2,.5) {$#1$};
    \end{tikzpicture}
}
\newcommand{\arrow}{
    \!\!\!\!\!\!\!
    \begin{tikzpicture}[scale=0.6]
        \node at (0,0) {};
        \node at (0,0.7) {$\longrightarrow$};
    \end{tikzpicture}
    \!\!\!\!\!\!\!
}

\title{A Short Proof of Zhang's Conjecture on Stack-Sortable Permutations}
\address{Massachusetts Institute of Technology}
\email{maxlu@mit.edu}
\begin{document}
\begin{abstract}
    We give a short proof of Zhang's conjecture that the number of stack-sortable permutations of length $n + 1$ which end in their smallest element is equal to the $n$-th Motzkin number.
\end{abstract}
\maketitle

\section{Introduction}

The stack-sorting map $s$ was defined by West~\cite{west}. It acts on the set of permutations a finite ordered set. The map is defined in terms of the following algorithm. Given an input $\pi$, we first write the $\pi$ on the right side of the stack, and the algorithm will push the elements of $\pi$ from the right side of the stack to the left. At each step, we pop elements from the stack to the left until the next element of $\pi$ on the right is less than the top element of the stack, or the stack is empty. Then push the next element of $\pi$ to the stack.
Once all elements of $\pi$ are processed, pop all elements from the stack.

The resulting sequence on the left-hand-side of the stack is the image $s(\pi)$ of $\pi$ under the stack-sorting map.

An example is given below with $\pi = 4231$.

\begin{figure}[h]
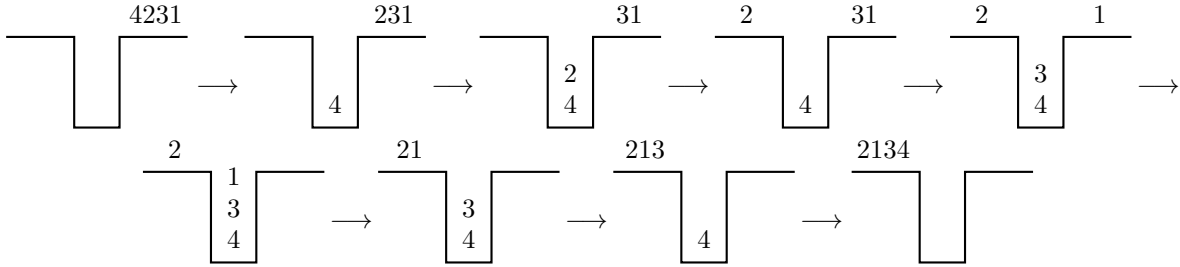

    \begin{center}
        \stack{}{}{}{}{4231}\arrow
        \stack{}{}{}{4}{231}\arrow
        \stack{}{}{2}{4}{31}\arrow
        \stack{2}{}{}{4}{31}\arrow
        \stack{2}{}{3}{4}{1}\arrow
        \stack{2}{1}{3}{4}{}\arrow
        \stack{21}{}{3}{4}{}\arrow
        \stack{213}{}{}{4}{}\arrow
        \stack{2134}{}{}{}{}
    \end{center}
    \caption{An example of the stack-sorting map with $\pi = 4231$, yielding $s(\pi) = 2134$.}
\end{figure}

Define $\langle \pi \rangle$ be the minimum value of $n$ such that $s ^ n(\pi)$ is sorted in increasing order. A permutation $\pi$ is $t$-stack-sortable if $\langle \pi \rangle \le t$.

One interesting area of study is the enumeration of permutations $\pi \in S_n$ which are $t$-stack sortable. Knuth~\cite{knuth} showed that the $1$-stack sortable permutations are exactly those which are $231$-avoiding, and that the number of such permutations is $C_n = \frac{1}{n + 1} \binom{2n}{n}$.
Zeilberger~\cite{zeilberger} showed that the number of $2$-stack-sortable permutations is equal to $\frac{2}{(n + 1)(2n + 1)}\binom{3n}{n}$.

In this paper, we focus on the behavior of the stack-sorting map on the subset of permutations which end with their smallest element, following work by Zhang~\cite{zhang2026enum}.

\begin{definition}
    Let $S_n'$ denote the set of permutations of $[n]$, followed by $0$, and let $\calW_t'(n)$ denote the set of $t$-stack-sortable permutations in $S_n'$.
\end{definition}

We recall the following definition of Motzkin paths, which are a generalization of Dyck paths.

\begin{definition}
    A Motzkin path of length $n$ is a path starting at $(0, 0)$ and ending at $(n, 0)$ consisting of up steps $(1, 1)$, down steps $(1, -1)$, and horizontal steps $(1, 0)$, which remain weakly above the line $y = 0$. The number of such paths is the $n$-th Motzkin number.
\end{definition}

We now state our main result, which was conjectured by Zhang~\cite[Conjecture 6.1]{zhang2026enum}.

\begin{theorem}\label{thm:motzkin}
    For all $n$, $\abs{\calW_2'(n)}$ is the $n$-th Motzkin number.
\end{theorem}

Inagaki and Luo~\cite{inagaki2026motzkinnumberscount2stacksortable} give a proof of~\cref{thm:motzkin} by finding a bijection from $\calW_2'(n)$ to Young tableaux with $n$ cells of width at most $3$, which are also enumerated by Motzkin numbers. Our proof, which was discovered independently, is much shorter and quite different, giving a bijection to Motzkin paths.

\section{Preliminaries}

We first restate a few definitions and results from~\cite{zhang2026enum}.
A composition of $n$ is an ordered sequence of positive integers with sum equal to $n$.
Let $\calC(n)$ denote the set of compositions of $n$.

\begin{definition}
    For a composition $\alpha = (\alpha_1, \dots, \alpha_\ell) \in \calC(n)$, the composition diagram of $\alpha$ is the set of cells
    \[D(\alpha) = \{(i, j) \mid 1 \le i \le \alpha_j\}.\] Let $w(\alpha) = \max \{\alpha_i\}$.
\end{definition}

We draw composition diagrams $D(\alpha)$ in English notation, with $(1, 1)$ at the upper-left.
We view $D(\alpha)$ as a poset, where $(a, b) \ge (c, d)$ if $a \le c$ and $b \le d$, so the poset increases up and to the left.
Let $\calL(D(\alpha))$ denote the set of linear extensions of $D(\alpha)$.

\begin{definition}
    For a composition diagram $D(\alpha)$, the hook lengths $h_\alpha(i, j)$ are defined by $h_\alpha(1, j) = 1$ and 
    \[h_\alpha(i, j) = \abs{\{(x, y) \in D(\alpha) \mid x \le i - 1, j' < y \le j\}},\] for $i > 1$, where
    \[j' = \max(\{y < j \mid \alpha_y \ge i - 1\} \cup \{0\}),\]
    that is, $j'$ is the row of the first cell above $(i - 1, j)$ in $D(\alpha)$, or $0$ if no such cell exists.
\end{definition}

\begin{ex}
    For $\alpha = (1, 3, 1, 4)$, the composition diagram $D(\alpha)$ is shown below (in English notation). Each cell is labelled by its hook length.
    \begin{figure}[h]
        \begin{center}
            \resizebox{0.2\textwidth}{!}{%
            \begin{tikzpicture}[every path/.style={draw, very thick}]
                \draw (0,0) rectangle ++(1,1) node[pos=.5] {1};
                \draw (1,0) rectangle ++(1,1) node[pos=.5] {1};
                \draw (2,0) rectangle ++(1,1) node[pos=.5] {3};
                \draw (3,0) rectangle ++(1,1) node[pos=.5] {4};
                \draw (0,1) rectangle ++(1,1) node[pos=.5] {1};
                \draw (0,2) rectangle ++(1,1) node[pos=.5] {1};
                \draw (1,2) rectangle ++(1,1) node[pos=.5] {1};
                \draw (2,2) rectangle ++(1,1) node[pos=.5] {3};
                \draw (0,3) rectangle ++(1,1) node[pos=.5] {1};
            \end{tikzpicture}
            }
        \end{center}
        \caption{The composition diagram $D(\alpha)$, where $\alpha = (1, 3, 1, 4)$. Each cell $(i, j)$ is labelled by $h_\alpha(i, j)$.}
        \label{fig:comp}
    \end{figure}
\end{ex}

We require the following result proven by Zhang~\cite{zhang2026enum}.

\begin{theorem}[\cite{zhang2026enum}]\label{thm:counting}
    There exists a map sending each $\pi \in S_n'$ to a pair $(\alpha_\pi, T_\pi)$ with $\alpha_\pi \in \calC(n)$ and $T_\pi \in \calL(D(\alpha_\pi))$ such that $\langle \pi \rangle = w(\alpha_\pi)$. Furthermore, for any $\alpha \in \calC(n)$ and $T \in \calL(D(\alpha))$, the number of permutations $\pi \in S_n'$ such that $(\alpha, T) = (\alpha_\pi, T_\pi)$ is
    \[\prod_{(i, j) \in D(\alpha)} h_\alpha(i, j).\]
\end{theorem}

Note that for any $\pi \in S_n'$, $\langle \pi \rangle \ge 1$, since $\pi$ ends with $0$.

The definition of the map $\pi \mapsto (\alpha_\pi, T_\pi)$ is given in~\cite[Section 3]{zhang2026enum}, and the enumeration of preimages is proven in~\cite[Theorem 4.10]{zhang2026enum}. We omit the details of the map, as they are unimportant for the proof of~\cref{thm:motzkin}.


\section{Proof of~\cref{thm:motzkin}}

We first prove a lemma on hook-lengths, then give a bijection to Motzkin paths using~\cref{thm:counting}.

\begin{lemma}\label{lem:hook-one}
    For any $\alpha \in \calC(n)$, if $w(\alpha) \le 2$, then $h_\alpha(i, j) = 1$ for all $(i, j) \in D(\alpha)$.
\end{lemma}

\begin{proof}
    Let $\alpha = (\alpha_1, \dots, \alpha_\ell)$. By definition, $h_\alpha(1, j) = 1$ for all $1 \le j \le \ell$.
    For any cell $(2, j) \in D(\alpha)$, we have
    \[h_\alpha(2, j) = \abs{\{(x, y) \in D(\alpha) \mid x \le 1, j' < y \le j\}},\] for $i > 1$, where $j' = \min(\{y \mid \alpha_y \ge i - 1\} \cup \{0\})$. Since $(1, j - 1) \in D(\alpha)$ for all $j \ge 2$, we have $j' = j - 1$. Thus, $h_\alpha(2, j) = 1$, as desired.
\end{proof}

\begin{proof}[Proof of~\cref{thm:motzkin}]
    By combining~\cref{thm:counting} and~\cref{lem:hook-one}, we have
    \[\abs{\calW_2'(n)} = \abs{\{(\alpha, T) \mid \alpha \in \calC(n), w(\alpha) \le 2, T \in \calL(D(\alpha))\}}.\]
    Thus, it suffices to compute the right-hand-side.

    We give a bijection between such pairs and Motzkin paths of length $n$.
    Consider any such pair $(\alpha, T)$.
    Partition the cells of $D(\alpha)$ into the following three classes:
    \begin{align*}
        S_0 &= \{(1, j) \mid \alpha_j = 1\}, \\
        S_{-1} &= \{(1, j) \mid \alpha_j = 2\}, \\
        S_1 &= \{(2, j) \mid \alpha_j = 2\},
    \end{align*}
    and define $T_i = T\rec(S_i)$.
    For each $1 \le r \le n$, if $r \in T_i$, then the $r$-th step in the corresponding path is $(1, i)$.

    Note that the $k$-th largest elements of $T_1$ and $T_{-1}$ are in the same row, and the cell of $T_{-1}$ is to the left, and thus has a larger entry. Hence, the $k$-th $(1, -1)$ step in the path comes after the $k$-th $(1, 1)$ step, so the path remains weakly above the $x$-axis.
    Furthermore, $\abs{S_1} = \abs{S_{-1}}$, so the path ends at $(n, 0)$. Thus, the path we construct is indeed a Motzkin path.

    We now describe the inverse map. Given a valid path from $(0, 0)$ to $(n, 0)$, iterate through the steps in reverse order. For the $r$-th step, going from $r = n$ to $r = 1$, perform the following procedure:
    \begin{itemize}
        \item If the step is $(1, 0)$ or $(1, -1)$, add a cell labelled $r$ to the bottom of the first column.
        \item If the step is $(1, 1)$, add a cell labelled $r$ to the right of the highest cell corresponding to a $(1, -1)$ step which does not already have a cell to its right.
    \end{itemize}
    By construction, the entries are increasing up and to the left, so this gives a valid pair $(\alpha, T)$.

    These two maps describe a bijection, so $\abs{\calW_2'(n)}$ is the $n$-th Motzkin number.
\end{proof}

For concreteness, we give an example of the above maps.
\begin{ex}
    For $n = 12$ and $\alpha = (2, 1, 2, 2, 1, 1, 2, 1)$, the pair $(\alpha, T)$ on the left corresponds to the path on the right.
    \begin{figure}[h]
        \begin{center}
            \resizebox{0.8\textwidth}{!}{%
                \begin{tikzpicture}[every path/.style={draw, very thick}, subfig/.style={}]
                    \node[subfig] (A) at (0,0){
                        \begin{tikzpicture}
                            \draw (0,0) rectangle ++(1,1) node[pos=.5] {1};
                            \draw (0,1) rectangle ++(1,1) node[pos=.5] {4};
                            \draw (1,1) rectangle ++(1,1) node[pos=.5] {2};
                            \draw (0,2) rectangle ++(1,1) node[pos=.5] {5};
                            \draw (0,3) rectangle ++(1,1) node[pos=.5] {7};
                            \draw (0,4) rectangle ++(1,1) node[pos=.5] {8};
                            \draw (1,4) rectangle ++(1,1) node[pos=.5] {3};
                            \draw (0,5) rectangle ++(1,1) node[pos=.5] {9};
                            \draw (1,5) rectangle ++(1,1) node[pos=.5] {6};
                            \draw (0,6) rectangle ++(1,1) node[pos=.5] {10};
                            \draw (0,7) rectangle ++(1,1) node[pos=.5] {12};
                            \draw (1,7) rectangle ++(1,1) node[pos=.5] {11};
                        \end{tikzpicture}
                    };
                    \node[subfig, anchor=west] (B) at (5, 0){
                        \begin{tikzpicture}[show background grid]
                            \draw (0,0)--(1,0)--(2,1)--(3,2)--(4,1)--(5,1)--(6,2)--(7,2)--(8,1)--(9,0)--(10,0)--(11,1)--(12,0);
                        \end{tikzpicture}
                    };
                    \draw[<->] ($(A.east)+(0.5,0)$)--($(B.west)-(0.5,0)$);
                \end{tikzpicture}
            }
        \end{center}
        \caption{An example of the bijection in the proof of~\cref{thm:motzkin}.}
        \label{fig:bijection}
    \end{figure}
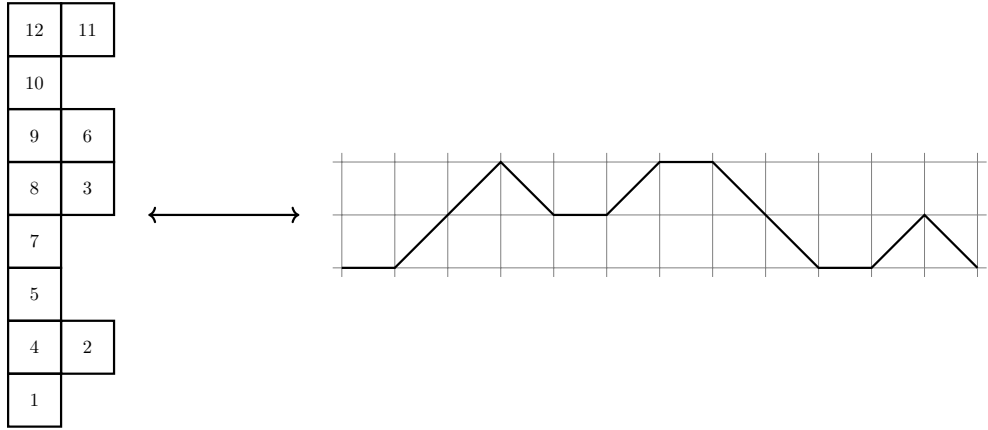
\end{ex}

\section{Acknowledgements}

This research was conducted at the University of Minnesota Duluth REU, which was funded by Jane Street Capital, Axiom Math AI, and a donation from Larry Penn. I am grateful to Joe Gallian and Colin Defant for organizing the REU and providing this opportunity.
I would also like to thank Eliot Hodges, Claire Kaneshiro, Noah Kravitz, and Carl Schildkraut for advising the REU.

\subsection{AI Usage Statement}

No AI assistance was used for this paper.

\bibliographystyle{alpha}
\bibliography{sources}

\end{document}